\documentclass[reqno,a4paper, 11pt]{amsart}

\usepackage[a4paper=true,pdfpagelabels]{hyperref}
\usepackage{graphicx}

\usepackage[ansinew]{inputenc}
\usepackage{amsfonts,epsfig}
\usepackage{latexsym}
\usepackage{mathabx}
\usepackage{amsmath}
\usepackage{amssymb}

\usepackage{color}

\newtheorem{theorem}{Theorem}
\newtheorem{lemma}[theorem]{Lemma}
\newtheorem{corollary}[theorem]{Corollary}

\newtheorem{proposition}[theorem]{Proposition}

\newtheorem{lettertheorem}{Theorem}

\theoremstyle{definition}

\theoremstyle{remark}
\newtheorem{remark}[theorem]{Remark}

\numberwithin{equation}{section}

\newcommand{\B}{\mathcal{B}}
\newcommand{\D}{\mathbb{D}}
\newcommand{\DD}{\widehat{\mathcal{D}}}

\newcommand{\DDD}{\mathcal{D}}

\newcommand{\RR}{\mathbb{R}}

\newcommand{\C}{\mathbb{C}}

\renewcommand{\phi}{\varphi}

\newcommand{\T}{\mathbb{T}}

\newcommand{\whw}{\widehat{\omega}}

\def\HL{\mathord{\rm HL}}

\def\T{\partial\D}
\def\f{\frac}

\def\at{(\varphi_t)_{t\geq0}}
\def\mt{(m_t)_{t\geq0}}
\def\St{(S_t)_{t\geq0}}
\def\Ct{(C_t)_{t\geq0}}

     \def\om{\omega}      
              \def\f{\frac}

\renewcommand{\H}{\mathcal{H}}

\begin{document}
\title{Weighted composition semigroups on some Banach spaces}


\author{Fanglei Wu}
\address{Department of Mathematics, Shantou University, Shantou, Guangdong 515063, China}
\address{University of Eastern Finland, P.O.Box 111, 80101 Joensuu, Finland}
\email{fangleiwu1992@gmail.com}

\thanks{The author was supported by NNSF of China (No. 11720101003) and NSF of Guangdong Province (No. 2018A030313512).}

\begin{abstract}
We characterize strong continuity of general operator semigroups on some Lebesgue spaces. In particular, a characterization of strong continuity of weighted composition semigroups on classical Hardy spaces and weighted Bergman spaces with regular weights is given. As applications, our result improves the results of Siskakis, A. G. \cite{AG1} and K\"{o}nig, W. \cite{K} and answers a question of Siskakis, A. G. proposed in \cite{AG4}. We also characterize strongly continuous semigroups of weighted composition operators on weighted Bergman spaces in terms of abelian intertwiners of multiplication operator $M_z$.
\end{abstract}

\maketitle

\section{Introduction}

Let $X$ denote a Banach space. A family $(T_t)_{t\geq0}$ of bounded linear operators acting on $X$ is called an operator semigroup if $T_0=I$ and $T_tT_s=T_{t+s}$ for all $t,s\geq0$. It is called strongly continuous if
$$
\lim_{t\rightarrow0^+}T_tf=f,\quad \mbox{for}~\mbox{any}~ f\in X.
$$
It is called weakly continuous if
$$
\lim_{t\rightarrow0^+}\langle T_tf, g\rangle=\langle f, g\rangle, \quad \mbox{for}~\mbox{any}~f\in X~ \mbox{and}~g\in X^*,
$$
where $X^*$ is the dual space of $X$ under $\langle\cdot\rangle$.
\par
Recall that an analytic semigroup of the unit disk $\D$ in the complex plane $\C$ is a family $\at$ of analytic self-maps of $\D$ if the following conditions hold.
\begin{itemize}
\item[(i)] $\varphi_{0}$ is the identity map of $\mathbb{D}$;
\item[(ii)] $\varphi_{t}\circ\varphi_{s}=\varphi_{t+s}$, for $t,s\geq0$;
\item[(iii)] for each $z\in\mathbb{D}$, $\varphi_{t}(z)\rightarrow z$, as $t\rightarrow0^+$.
\end{itemize}
\par
The infinitesimal generator of $\at$ is defined as the function
$$
G(z)=\lim_{t\mapsto0^+}\f{\varphi_t(z)-z}{t},\,\, z\in\D.
$$
\par
 Let $H(\D)$ be the set of analytic functions on $\D$. Given a semigroup of analytic functions $\at$, a family $\mt$ with $m_t(z)\in H(\D)$ is said to be a cocycle for $\at$ if it satisfies:
\begin{itemize}
\item[(i)] $m_{t+s}(z)=m_t(z)m_s(\varphi_t(z))$ for all $t,s\geq0$ and $z\in\D$;
\item[(ii)] $m_0(z)=1$ for all $z\in\D$.
\end{itemize}
\par
 For a non-vanishing analytic function $w(z)$ and a semigroup of analytic functions $\at$, if all zeros of $w(z)$ are in the set $\{z\in\D:\varphi_t(z)=z,~\mbox{for}~\mbox{all}~t\in[0,\infty)\}$, the fixed points of $\at$, then $m_t(z)$ defined as follows:
$$
m_t(z)=\f{w(\varphi_t(z))}{w(z)}\quad~~\mbox{for}~~\mbox{all}~~t\geq0~~\mbox{and}~~z\in\D
$$
is a cocycle for $\at$, which is generally said to be a coboundary of $\at$. In particular, if we choose $w(z)=G(z)$, then
$$
\f{w(\varphi_t(z))}{w(z)}=\varphi_t'(z).
$$
Not all cocycles are coboundaries. See \cite{EJK,JTT,PT} for more information about cocycles.
\par
Given a semigoup $\at$ and a cocycle $\mt$ for $\at$, the formula
$$
(C_tf)(z)=f(\varphi_t(z)), \quad f\in X
$$
defines a semigroup $(C_t)_{t\geq0}$ of composition operators on $X$, provided that each $C_t$ is bounded on $X$;
the formula
\begin{equation}\label{Def1}
(S_tf)(z)=m_t(z)f(\varphi_t(z)), \quad f\in X
\end{equation}
defines a semigroup $\St$ of weighted composition operator on $X$, provided that each $S_t$ is bounded on $X$.
\par
In 1978, Berkson, E. and Porta, H. \cite{BP} initially studied strong continuity of semigroups of composition operators acting on the classical Hardy space $H^p$, $1\leq p<\infty$. They proved that for $1\leq p<\infty$ $\Ct$ is strongly continuous on $H^p$. Later, Siskakis, A. G. demonstrated that $\Ct$ is strongly continuous on the Bergman space $A_{\alpha}^p$ ($1\leq p<\infty$, $-1<\alpha<\infty$) and the Dirichlet space $\mathcal{D}$ in \cite{AG2} and \cite{AG3}. However, strong continuity of weighted composition semigroups on $H^p$ is much more complicated than unweighted cases. As a matter of fact, strong continuity of weighted composition semigroup $\St$ relates closely to the cocycle $\mt$. If $\mt$ is a coboundary, Siskakis, A. G. found some sufficient conditions of $\mt$ such that $\St$ is strongly continuous on $H^p$, $1\leq p<\infty$ in \cite{AG1}.

\begin{lettertheorem}[\cite{AG1}]\label{A}
 Let $1\leq p<\infty$. Suppose $\at$ is a semigroup of analytic self-maps of $\D$ and $w(z)\in H(\D)$ is non-vanishing with all zeros in the set of the fixed points of $\at$. Then either of the following two conditions implies strong continuity of weighted composition semigroup defined as \eqref{Def1} on $H^p$:
\begin{equation*}
\limsup_{t\rightarrow0^+}\|\f{w\circ\varphi_t}{w}\|_{H^\infty}\leq1;
\end{equation*}
\begin{equation*}
w(z)\in H^q ~~\mbox{for}~\mbox{some}~q>0\quad \mbox{and}\quad\limsup_{t\rightarrow0^+}\|\f{w\circ\varphi_t}{w}\|_{H^\infty}<\infty.
\end{equation*}
\end{lettertheorem}

In \cite{K}, K\"{o}nig, W. extended these results to arbitrary cocycle $\mt$.

\begin{lettertheorem}[\cite{K}]\label{B}
Let $1\leq p<\infty$. Suppose $\at$ is a semigroup of analytic self-maps of $\D$ and $\mt$ is a cocycle for $\at$. If
\begin{equation*}
\limsup_{t\rightarrow0^+}\|m_t(z)\|_{H^\infty}\leq1,
\end{equation*}
then the weighted composition semigroup $\St$ defined as \eqref{Def1} is strongly continuous on $H^p$.
\end{lettertheorem}
For more results about weighted composition semigroups, see \cite{AM,AO,Ka} and the reference therein. \par
However, the question of characterizing strong continuity of weighted composition semigroups on Hardy spaces remains open. Our first main result characterizes the strong continuity of general operator semigroups on some Lebesgue spaces. As an application, we provide a characterization of strong continuity of weighted composition semigroups on $H^p$ ($1<p<\infty$) and the weighted Bergman space $A^p_{\omega}$ ($1<p<\infty$) with regular weight $\omega$. Therefore, the above question in the case of $1<p<\infty$ can be answered. Moreover, for $1\leq p<\infty$, the sufficient conditions presented in Theorem A and Theorem B can be improved.

The other main result focuses on depicting strongly continuous semigroups of weighted composition operators on $A^p_{\alpha}$ ($1\leq p<\infty$, $-1<\alpha<\infty$), the weighted Bergman space with typical weight $(\alpha+1)(1-|z|)^{\alpha}$. The method used here is motivated by the method given by Jafari, F., Slodkowski, Z. and Tonev, T. who provided a connection between weighted composition operators and abelian intertwiners of multiplication operator $M_z$ on $H^p$ ($1\leq p<\infty$) in \cite{JST}.

The paper is organized as follows. In section 2, after some preliminaries, we prove our first main result, see Theorem \ref{thm1}. Section 3 contains several applications of Theorem \ref{thm1}. In section 4, via verifying a connection between abelian intertwiners of $M_z$ and weighted composition operators on $A^p_{\alpha}$, as the second main result, Theorem \ref{thm2} establishes an equivalence of strongly continuous weighted composition semigroups on $A^p_{\alpha}$.

Throughout the paper $\f1p+\f{1}{p'}=1$ for $1<p<\infty$, and the symbol $A\approx B$ means that $A\lesssim B\lesssim
A$. We say that $A\lesssim B$ if there exists a constant $C$ such that $A\leq CB$.

\section{Strongly Continuous Operator Semigroups on $L^p(\Omega,\mu)$}
In this section, we characterize strong continuity of operator semigroups on Lebesgue spaces. Recall that $(\Omega,\Sigma,\mu)$ is a measure space consisting of a set $\Omega$, a $\sigma$-algebra $\Sigma$ of subsets of $\Omega$, and a countably additive measure $\mu$ defined on $\Sigma$ with values in the non-negative extended real numbers. For $0<p<\infty$ the Lebesgue space
$$
L^p(\Omega,\mu):=L^p(\Omega,\Sigma,\mu)
$$
consists of all $p$-integrable complex functions on $\Omega$, i.e.
$$
\|f\|_p:=\left(\int_{\Omega}|f(z)|^pd\mu(z)\right)^{\f{1}{p}}<\infty.
$$
It is well-known that for $1\leq p<\infty$ $L^p(\Omega,\mu)$ is a Banach space endowed with the above norm and for $1<p<\infty$ the dual space of $L^p(\Omega,\mu)$ can be identified with $L^{p'}(\Omega,\mu)$ under the pairing:
$$
\langle f,g\rangle:=\int_{\Omega}f(z)\overline{g(z)}d\mu(z),
$$
where $f\in L^p(\Omega,\mu)$ and $g\in L^{p'}(\Omega,\mu)$. The following lemma is critical for the first result.

\begin{lemma}[\cite{C}]\label{lemma1}
Let $1<p<\infty$ and $(\Omega,\mu)$ be a measure space. For a bounded sequence $\{f_n\}^\infty_0$ in $L^p(\Omega,\mu)$, if there is an $f\in L^p(\Omega,\mu)$ such that $\lim_{n\rightarrow\infty}f_n(z)=f(z)$ a.e., then
$$
\lim_{n\rightarrow\infty}\langle f_n, g\rangle=\langle f,g\rangle
$$
for any $g\in L^{p'}(\Omega,\mu)$.
\end{lemma}

Under a simple assumption, we are able to characterize strong continuous of operator semigroups on the Lebesgue space, which can be read as follows:

\begin{theorem}\label{thm1}
Let $1<p<\infty$, and let $(T_t)_{t\geq0}$ be an operator semigroup on $L^p(\Omega,\mu)$. If the point evaluation functionals on $L^p(\Omega,\mu)$ are bounded, then $(T_t)_{t\geq0}$ is strongly continuous on $L^p(\Omega,\mu)$ if and only if
$$
\sup_{0\leq t<1}\|T_t\|<\infty.
$$
\end{theorem}

\begin{proof}
Assume $(T_t)_{t\geq0}$ is a strongly continuous operator semigroup on $L^p(\Omega,\mu)$. Then by \cite[Proposition 1.4]{EN}, we immediately get
$$
\sup_{0\leq t<1}\|T_t\|<\infty.
$$
\par
Conversely, suppose the set $\{\|T_t\|: t\in[0,1)\}$ is bounded. Fix an $f\in L^p(\Omega,\mu)$ and a sequence $t_n\rightarrow0$ as $n\rightarrow\infty$. Since any point evaluation functional is bounded,
$$
\lim_{n\rightarrow\infty}T_{t_n}f(z)=f(z).
$$
By Lemma~\ref{lemma1}, for any $g\in L^{p'}(\Omega,\mu)$, we have
$$
\lim_{n\rightarrow\infty}\langle T_{t_n}f, g\rangle=\langle f,g\rangle.
$$
Therefore,
$$
\lim_{t\rightarrow0^+}\langle T_{t}f, g\rangle=\langle f,g\rangle
$$
for any $g\in L^{p'}(\Omega,\mu)$, which means that
$(T_t)_{t\geq0}$ is weakly continuous on $L^p(\Omega,\mu)$. Hence it is strongly continuous on  $L^p(\Omega,\mu)$ by \cite[Theorem 1.4]{P}. Thus the proof is complete.
\end{proof}

For $0<p<\infty$, the Hardy space $H^p$ consists of $f\in H(\D)$ such that
$$
\|f\|^p_{H^p}=\sup_{0<r<1}\f{1}{2\pi}\int_0^{2\pi}|f(re^{i\theta})|^pd\theta.
$$
For any $f\in H^p$, $0<p<\infty$, the radial limit of $f$ exists and
$$
\|f\|^p_{H^p}=\f{1}{2\pi}\int_0^{2\pi}|f(e^{i\theta})|^pd\theta.
$$
When $p=\infty$,  $H^{\infty}$ denotes the space of all bounded analytic functions on $\D$.
Moreover, for $1<p<\infty$, the dual space of $H^p$ is $H^{p'}$ under the foregoing pairing
$$
\langle f,g\rangle=\f{1}{2\pi}\int_0^{2\pi}f(e^{i\theta})\overline{g(e^{i\theta})}d\theta,
$$
where $f\in H^p$ and $g\in H^{p'}$.

A function $\omega:\D\rightarrow [0,\infty)$, integrable over $\D$
is called a weight. We say that $\omega$ is radial if $\omega(z)=\omega(|z|)$ for all $z\in\D$. For $0<p<\infty$ and a weight $\omega$ in $\D$, the weighted Bergman space $A_{\omega}^p$ consists of all $f\in H(\D)$ such that
$$
\|f\|_{A_{\omega}^p}=\left(\int_{\D}|f(z)|^{p}\omega(z)dA(z)\right)^{\frac1p}.
$$
where $dA(z)=\frac{1}{\pi}dxdy$ is the normalized Lebesgue measure on $\D$. We say that a radial weight $\omega$ is regular if
$$
\int^1_{r}\omega(s)ds\approx \omega(r)(1-r), \quad\mbox{for}~\mbox{all}~0\leq r<1.
$$
A typical example of regular weight is the standard weight $(\alpha+1)(1-|z|^2)^{\alpha}$ with $-1<\alpha<\infty$ and we write the weighted Bergman space $A^p_{\alpha}$. See more results about weighted Bergman spaces in \cite{PR1,Zhu}. By \cite[Corollary 7]{PR}, if $1<p<\infty$ and $\omega$ is a regular weight, then the dual space of $A^p_{\omega}$ can be identified with $A^{p'}_{\omega}$ via the pairing
$$
\langle f,g\rangle=\int_{\D}f(z)\overline{g(z)}\omega(z)dA(z)
$$
for every $f\in A^p_{\omega}$ and $g\in A^{p'}_{\omega}$.

It is well-known that both $H^p$ and $A^p_{\omega}$ can be treated as two subspaces of $L^p(\Omega,\mu)$ with special $\Omega$ and $\mu$. Moreover, each point evaluation functional is bounded on both $H^p$ and $A^p_{\omega}$. Therefore, bearing in mind the dual of $H^p$ and $A^p_{\omega}$, it is easy to obtain the following results by Theorem \ref{thm1}.
\begin{corollary}\label{corhp}
Let $1<p<\infty$, and let $(T_t)_{t\geq0}$ be an operator semigroup on $H^p$. Then $(T_t)_{t\geq0}$ is strongly continuous if and only if
$$
\sup_{0\leq t<1}\|T_t\|<\infty.
$$
\end{corollary}

\begin{corollary}\label{corap}
Let $1<p<\infty$, $\omega$ be a regular weight, and let $(T_t)_{t\geq0}$ be an operator semigroup on $A^p_\omega$. Then $(T_t)_{t\geq0}$ is strongly continuous if and only if
$$
\sup_{0\leq t<1}\|T_t\|<\infty.
$$
\end{corollary}

\begin{remark}
Since the weak convergence in Lemma \ref{lemma1} does not work for the case $p=1$, the above results do not hold for the case $p=1$. Apparently, if $(T_t)_{t\geq0}$ is the weighted composition semigroup defined as \eqref{Def1}, then the above corollaries give an abstract characterization of strong continuity of weighted composition semigroups on $H^p$ and $A^p_{\omega}$ ($1<p<\infty$ and $\omega$ is regular) respectively. In addition, the similar characterization can be obtained for other Banach space of analytic functions on $\D$ if $m_t(z)=\varphi'_t(z)$. For instance, the author of \cite{St} provided the similar characterization on $VMOA$, the space of analytic functions of vanishing mean oscillation.
\end{remark}

\section{Weighted Composition Semigroup on $H^p$ and $A^p_{\omega}$}
In this section, we are going to apply foregoing results to weighted composition semigroups on the Hardy space $H^p$ and the weighted Bergman space $A^p_{\omega}$ with regular weight $\omega$. By Corollary \ref{corhp} and \ref{corap}, to characterize strong continuity of weighted composition semigroups on $H^p$ and $A^p_{\omega}$, it is critical to characterize the boundedness of weighted composition operators on these spaces. Now, let $m(z)\in H(\D)$, and let $\varphi$ be an analytic self-map of $\D$. For convenience, the weighted composition operator can be written as
\begin{equation}\label{Def2}
M_mC_{\varphi}f(z)=m(z)f(\varphi(z))
\end{equation}
for $f\in H(\D)$. Here $M_m$ is known as multiplication operator which is defined by
$$
M_mf(z)=m(z)f(z),\quad f\in H(\D).
$$

Now, we are in a position to characterize strong continuity of weighted composition semigroups on $H^p$ and $A^p_{\omega}$ as long as $1<p<\infty$ and $\omega$ is regular.

\begin{theorem}\label{thmhp}
Let $1<p<\infty$. Suppose $\at$ is a semigroup of analytic self-maps of $\D$ and $\mt$ is a cocycle for $\at$. Then the weighted composition semigroup $(S_t)_{t\geq0}$ defined as \eqref{Def1} is strongly continuous on $H^p$ if and only if
\begin{equation}
\sup_{0\leq t<1}\sup_{a\in\D}\int_{\partial\D}\f{(1-|a|^2)|m_t(z)|^p}{|1-\bar{a}\varphi_t(z)|^2}d\sigma(z)<\infty,
\end{equation}
where $d\sigma(z)$ is the normalized length measure of $\T$
\end{theorem}
\begin{proof}
For any $t\in[0,\infty)$, by \cite[Proposition 1]{CZ} for example, we see that $M_{m_t}C_{\varphi_t}$ is bounded on $H^p$ if and only if
$$
\Delta_{p,t}:=\sup_{a\in\D}\int_{\partial\D}\f{(1-|a|^2)|m_t(z)|^p}{|1-\bar{a}\varphi_t(z)|^2}d\sigma(z)<\infty.
$$
Moreover, $\|M_{m_t}C_{\varphi_t}\|\thickapprox\Delta_{p,t}$. It follows from Corollary \ref{corhp} that $(S_t)_{t\geq0}$
is strongly continuous on $H^p$ if and only if
\begin{equation*}
\sup_{0\leq t<1}\sup_{a\in\D}\int_{\partial\D}\f{(1-|a|^2)|m_t(z)|^p}{|1-\bar{a}\varphi_t(z)|^2}d\sigma(z)<\infty.
\end{equation*}
Hence the proof is complete.
\end{proof}

If the cocycle $\mt$ for $\at$ is coboundary or more special $(\varphi_t')_{t\geq0}$, then Theorem \ref{thmhp} answers a question of \cite{AG3} for $1<p<\infty$. Moreover, as a corollary, Theorem A and B can be improved.

\begin{corollary}\label{corhinfty}
 Let $1\leq p<\infty$. Suppose $\at$ is a semigroup of analytic self-maps of $\D$ and $\mt$ is a cocycle for $\at$. If
\begin{equation}\label{eq1}
\limsup_{t\rightarrow0^+}\|m_t(z)\|_{H^\infty}<\infty,
\end{equation}
then the weighted composition semigroup $\St$ defined as \eqref{Def1} is strongly continuous on $H^p$.
\end{corollary}

\begin{proof}
If  $\limsup_{t\rightarrow0^+}\|m_t(z)\|_{\infty}<\infty$, then by Theorem \ref{thmhp}, we immediately obtain that the weighted composition semigroup is strongly continuous on $H^p$ for all $1<p<\infty$. Therefore, it remains to show the case when $p=1$. As a matter of fact, if \eqref{eq1} holds, then there is a positive constant $M$ such that
$$
\sup_{0\leq t<1}\|S_t\|\leq M.
$$
Notice that the set of all polynomials $P$ is dense in both $H^1$ and $H^2$. Therefore, for any $\varepsilon>0$ and $f\in H^1$, there exists a polynomial $g\in P$ satisfying $\|g-f\|_{H^1}<\varepsilon$. Moreover, since $(S_t)_{t\geq0}$ is strongly continuous on $H^2$, for the above $\varepsilon$ there exists $\delta=\delta(\varepsilon)<1$ such that for all $t\in(0,\delta)$, $\|S_tg-g\|_{H^2}<\varepsilon$. Consequently, for $t\in(0,\delta)$, triangle inequality and H\"{o}lder inequality yield
\begin{equation*}
\begin{split}
\|S_tf-f\|_{H^1}&\leq\|S_tf-S_tg\|_{H^1}+\|S_tg-g\|_{H^1}+\|g-f\|_{H^1}\\
&\leq(\|S_t\|+1)(\|g-f\|_{H^1})+\|S_tg-g\|_{H^2}\\
&\leq(M+1)\|g-f\|_{H^1}+\|S_tg-g\|_{H^2}\lesssim\varepsilon,
\end{split}
\end{equation*}
which indicates that $(S_t)_{t\geq0}$ is strongly continuous on $H^1$. Hence we complete the proof.
\end{proof}

To characterize strong continuity of weighted composition semigroups on $A^p_{\omega}$, we need to consider the Carleson square $S(I)$, which is the set
$$
S(I)=\{re^{it}\in\D: e^{it}\in I, 1-|I|\leq r<1\},
$$
where $I$ is an interval on $\T$ and $|I|$ denotes the Lebesgue measure of $I$. As usual, we define for each $a\in\D\backslash\{0\}$ the interval
$$
I_a=\{e^{i\theta}:|\arg(ae^{-i\theta})|\leq\f{1-|a|}{2}\}
$$
and denote $S(a)=S(I_a)$.
By \cite{PR1}, if $\omega$ is regular, then for any $0<p<\infty$ one can find an enough large $\gamma>0$ such that
$$
\sup_{a\in\D}\|f_{a,p}\|_{A_{\omega}^p}<\infty,
$$
where
\begin{equation}
f_{a,p}(z)=\f{(1-|a|)^{\f{\gamma+1}{p}}}{(1-\bar{a}z)^{\f{\gamma+1}{p}}\omega(S(a))^{\f1p}}\quad\mbox{and}\quad
\omega(S(a))=\int_{S(a)}\omega(z)dA(z).
\end{equation}

\begin{theorem}
Let $1<p<\infty$ and $\omega(z)$ be a regular weight. Suppose $\at$ is a semigroup of analytic self-maps of $\D$ and $\mt$ is a cocycle for $\at$. Then the weighted composition semigroup defined as \eqref{Def1} is strongly continuous on $A^p_{\omega}$ if and only if
$$
\sup_{0\leq t<1}\sup_{a\in\D}\int_{\D}|f_{a,p}(\varphi_t(z))|^p|m_t(z)|^p\omega(z)dA(z)<\infty.
$$
\end{theorem}
\begin{proof}
For any $t\in[0,\infty)$, it follows from \cite[Theorem 1]{DLS} that $M_{m_t}C_{\varphi_t}$ is bounded on $A^p_{\omega}$ if and only if
$$
\Delta_{\omega,p,t}:=\sup_{a\in\D}\int_{\D}|f_{a,p}(\varphi_t(z))|^p|m_t(z)|^p\omega(z)dA(z)<\infty.
$$
Moreover, $\|M_{m_t}C_{\varphi_t}\|\thickapprox\Delta_{\omega,p,t}$. Therefore, Corollary \ref{corap} shows that $(S_t)_{t\geq0}$
is strongly continuous on $A^p_{\omega}$ if and only if
$$
\sup_{0\leq t<1}\sup_{a\in\D}\int_{\D}|f_{a,p}(\varphi_t(z))|^p|m_t(z)|^p\omega(z)dA(z)<\infty.
$$
\end{proof}

Since the set of all polynomials is dense in the weighed Bergman space $A^p_{\omega}$ as long as $w$ is a radial weight, by a similar proof of as Corollary \ref{corhinfty} we have:

\begin{corollary}
Let $1\leq p<\infty$ and $\omega(z)$ be a regular weight. Suppose $\at$ is a semigroup of analytic self-maps of $\D$ and $\mt$ is a cocycle for $\at$. If
\begin{equation*}
\limsup_{t\rightarrow0^+}\|m_t(z)\|_{H^\infty}<\infty,
\end{equation*}
then the weighted composition semigroup $\St$ defined as \eqref{Def1} is strongly continuous on $A^p_{\omega}$.
\end{corollary}

\section{Abelian Intertwiner and Weighted Composition Semigroup on $A^p_{\alpha}$}

Jafari, F., Slodkowski, Z. and Tonev, T. in \cite{JST} provided a characterization of strong continuity of weighted semigroups on the $H^p$ ($1\leq p<\infty$) through using abelian intertwiners of multiplication operator $M_z$ on $H^p$. Motivated by their results, in this section, we consider the similar problem on the weighted Bergman space $A^p_{\alpha}$ for $1\leq p<\infty, -1<\alpha<\infty$.

Let $B(X)$ be the set of all linear bounded operators on a Banach space $X$. We say that an operator $R$ belongs to the {commutant algebra} of $A\in B(X)$, denote by $A'$, if $R\in B(X)$ and $AR=RA$. We say that $T\in B(X)$ is an {abelian intertwiner} of $A$ if there exists an operator $B\in A'$ such that $TA=BT$.

The following lemma can be proved by a simple combination of \cite[Corollary 5]{RS} and \cite[Theorem 1]{zhao}.
\begin{lemma}\label{lemma2}
Let $1\leq p<\infty$ and $-1<\alpha<\infty$. The commutant algebra $M_z'$ of the operator $M_z$ on $A^p_{\alpha}$ is $\{M_f:f\in H^{\infty}\}$.
\end{lemma}

The following consequence is key to prove the second main result of this paper.

\begin{proposition}\label{pro}
Let $1\leq p<\infty$, $-1<\alpha<\infty$, and let $\varphi$ be an analytic self-map of $\D$. If $m(z)\in H^{\infty}$, then $T=M_mC_{\varphi}$ is an abelian intertwiner of the multiplication operator $M_z$ on $A_{\alpha}^p$.
\par Conversely, if $T\in B(A^p_{\alpha})$ is an abelian intertwiner of $M_z$, and further $(T1)(z)\in H^{\infty}$, then $T$ is a weighted composition operator of the form \eqref{Def2} satisfying $m(z)\in H^{\infty}$.
\end{proposition}
\begin{proof}
\par
If $T=M_mC_{\varphi}$, for every $f\in A^p_{\alpha}$, we have
\begin{equation*}
\begin{split}
(TM_zf)(z)&=(T(zf))(z)=(M_mC_{\varphi}(zf))(z)\\
&=m(z)\varphi(z)(f\circ\varphi)(z)=\varphi(z)m(z)(f\circ\varphi)(z)\\
&=\varphi(z)(M_mC_{\varphi}f)(z)=(M_{\varphi}Tf)(z).
\end{split}
\end{equation*}
So $TM_z=M_{\varphi}T$. In addition, for every $f\in A^p_{\alpha}$, $M_zM_{\varphi}f=z\varphi f=\varphi zh=M_{\varphi}M_zf$, which indicates that $M_{\varphi}\in M_z'$. Since $m(z)\in H^{\infty}$, $T$ is bounded on $A^p_{\alpha}$. Hence, $T$ is an abelian intertwiner of $M_z$.
\par
Conversely, if $T\in A^p_{\alpha}$ is an abelian intertwiner of $M_z$, the proof is trivial when $T=0$. Now, suppose $T\neq0$. By Lemma \ref{lemma2}, there exists a $\varphi\in H^{\infty}$ such that $TM_z=M_{\varphi}T$. By induction, it is elementary to see that $TM_z^n=M_{\varphi}^nT$, $n\in\mathbb{N}$. Since $T\neq0$, there is an $f\in A^p_{\alpha}$ satisfying $g=Tf\neq0$, and hence $TM_z^nf=M_{\varphi}^nTf=M_{\varphi}^ng$. Moreover, we have
$$
\|M_{\varphi}^nTf\|_{A^p_{\alpha}}=\|TM_z^nf\|_{A^p_{\alpha}}\leq\|T\|\|M_z^nf\|_{A^p_{\alpha}}\leq
\|T\|\|f\|_{A^p_{\alpha}},\quad \mbox{for} ~n\in\mathbb{N}.
$$
Applying \cite[Theorem 4.14]{Zhu}, we get
\begin{equation*}
\begin{split}
\|T\|\|f\|_{A^p_{\alpha}}\geq\|M_{\varphi}^nTf\|_{A^p_{\alpha}}&=\|\varphi^nTf\|_{A^p_{\alpha}}\\
&\geq(1-|z|^2)^{\f{2+\alpha}{p}}|\varphi(z)|^n|Tf(z)|,\quad z\in\D.
\end{split}
\end{equation*}
Therefore, for each $z\in\D$, we have
$$
|\varphi(z)|\leq\left(\frac{\|T\|\|f\|_{A^p_{\alpha}}}{|g(z)|(1-|z|^2)^{\frac{2+\alpha}{p}}}\right)^{\frac{1}{n}}.
$$
By letting $n$ go to $\infty$, we deduce that $|\varphi(z)|\leq1$, $z\in\D$ a.e. Since $\varphi\in H^\infty$, it follows from maximum modulus principle that $|\varphi(z)|\leq1$ for all $z\in\D$. Moreover, open mapping theorem indicates that either $|\varphi(z)|<1$  for all $z\in\D$ or $\varphi(z)\equiv \zeta$ for a fixed $\zeta\in\T$. However, the second case cannot happen.
\vskip 0.3cm
{\bf Claim}.~~$|\varphi(z)|<1$  for all $z\in\D$.
\begin{proof}
Suppose not. Then there exists a $\zeta\in\T$ such that $\varphi\equiv \zeta$. When $1<p<\infty$, since $T\neq0$, there is an $f\in A^p_{\alpha}$ such that $g=Tf\neq0$.
Consider the bounded sequence $\{h_n(z)\}$ in $A^p_{\alpha}$, where $h_n(z)=z^nf(z)$. Since $A^p_{\alpha}$ is a reflexive space, there is a subsequence $\{h_{n_k}\}$ such that it converges weakly  in $A^p_{\alpha}$. Montel's Theorem tells us that there is a subsequence of $\{h_{n_k}\}$  pointwisely convergent to 0. Therefore, $h_{n_k}(z)$ is convergent weakly to 0. Nevertheless, from the above statement, we have
$TM_z^{n_k}(f)(z)=\zeta^{n_k}g(z)$. That is, we obatin
$$
|g(z)|=\lim_{k\rightarrow\infty}|\zeta^{n_k}g(z)|=\lim_{k\rightarrow\infty}|TM_z^{n_k}(f)(z)|=\lim_{k\rightarrow\infty}
|T(h_{n_k})(z)|=|T(0)(z)|=0.
$$
This is a contradiction.
\par
When $p=1$, we have $TM_z=M_{\varphi}T=\zeta T$. Since $T\in B(A^1_{\alpha})$, we get for any $f\in A^1_{\alpha}$,
$$
T(zf-\zeta f)=0.
$$
But take into account the fact that $A=\{zf-\zeta f:f\in A^1_{\alpha}\}$ is dense in $A^1_{\alpha}$, which indicates that $T$ is not only vanishing on $A$ but also on the whole $A^1_{\alpha}$. This is contradicted to the assumption $T\neq0$. Thus the claim is proved.
\end{proof}
\par
Now, for $1\leq p<\infty$ and $-1<\alpha<\infty$, if $T\in A^p_{\alpha}$ is an abelian intertwiner of $M_z$, then there exists an analytic function $\varphi$ from $\D$ to $\D$ such that $TM_z=M_{\varphi}T$. Let $a\in\D$ and $f\in A^p_{\alpha}$. Consider the function
\begin{equation*}
F(z)=
\left\{
\begin{aligned}
&\f{f(z)-f(\varphi(a))}{z-\varphi(a)}\quad z\neq\varphi(a),\\
&f'(\varphi(a))\quad z=\varphi(a).
\end{aligned}
\right.
\end{equation*}
Since $\varphi(a)\subseteq\D$, $F(z)\in A^p_{\alpha}$. Moreover,
$$
f-f(\varphi(a))=M_zF-\varphi(a)F.
$$
Accordingly,
$$
Tf-f(\varphi(a))T1=TM_zF-\varphi(a)TF=M_{\varphi}TF-\varphi(a)TF.
$$
Then
$$
(Tf)(a)=f(\varphi(a))T1(a)=(M_{T1(a)}C_{\varphi}f)(a):=(M_mC_{\varphi}f)(a),
$$
where $m(z)=T1(z)$. Since $a$ is arbitrary in $\D$, $Tf=M_mC_{\varphi}$. Hence $T$ is a weighted composition operator.
\end{proof}

We end up this paper by presenting the second main result, which gives a new result concerning the equivalence of the strongly continuous weighted composition semigroup on $A^p_{\alpha}$.

\begin{theorem}\label{thm2}
Let $1\leq p<\infty$, $-1<\alpha<\infty$. If $(T_t)_{t\geq0}$ is an operator semigroup on $A^p_{\alpha}$ satisfying
$$
\sup_{0\leq t<1}\|T_t\|<\infty,
$$
then $(T_t)_{t\geq0}$ is a weighted composition semigroup in the form of \eqref{Def1} if and only if each $T_t$ is an abelian intertwiner of $M_z$.
\end{theorem}
\begin{proof}
The necessity is easy to obtain by using the first part of Proposition \ref{pro}. Conversely, since
$$
\sup_{0\leq t<1}\|T_t\|<\infty,
$$
Corollary \ref{corap} implies that $(T_t)_{t\geq0}$ is strongly continuous on $A^p_{\alpha}$. If each $T_t$ is an abelian intertwiner of $M_z$, then there is an analytic $\varphi_t$ from $\D$ to $\D$ such that for every $f\in A^p_{\alpha}$, $T_tf=M_{m_t}C_{\varphi_t}f$ by the proof of Proposition \ref{pro}. Moreover, $m_t=T_t1$. If $f(z)=z$, then $\varphi_t(z)=\f{T_tz}{T_t1(z)}$. Hence $\varphi_0$ is the identity map. Since $(T_t)_{t\geq0}$ is an operator semigroup, we have $m_0\equiv1$ and for every $f\in A^p_{\alpha}$, $T_{t+s}f=T_tT_sf$. Consequently,
\begin{equation*}
\begin{split}
&T_{t+s}f=M_{m_{t+s}}C_{\varphi_{t+s}}f=m_{t+s}f\circ\varphi_{s+t}=T_tT_sf\\
&=M_{m_t}C_{\varphi_t}(m_sf\circ\varphi_s)=m_t(m_s(f\circ\varphi_s))\circ\varphi_t=m_t(m_s\circ\varphi_t)f\circ
(\varphi_s\circ\varphi_t).
\end{split}
\end{equation*}
Accordingly,
$$
m_{t+s}f\circ\varphi_{s+t}=m_t(m_s\circ\varphi_t)f\circ(\varphi_s\circ\varphi_t).
$$
Let $f=1$. We conclude that $m_{s+t}=m_t(m_s\circ\varphi_t)$. Also, we claim that $m_t(z)\neq0$. Otherwise, there is the smallest $\tau>0$ and $z_0\in\D$, such that $m_{\tau}(z_0)=0$. By the identity $m_{s+t}=m_t(m_s\circ\varphi_t)$,
$$
m_{\tau}(z)=m_s(z)m_{\tau-s}(\varphi_{\tau}(z)),\quad 0<s<\tau.
$$
Together with the choice of $\tau$, we deduce that $m_{\tau-s}(\varphi_{\tau}(z_0))=0$ for $0<s<\tau$, which implies that $m_t(\varphi_{\tau}(z_0))=0$. Let $t\rightarrow0$. We get the contradiction of the fact $m_0(z)\equiv1$. Furthermore,
\begin{align*}
m_{t+s}(z)\varphi_{t+s}(z)&=T_{t+s}z=T_t(m_s(z)\varphi_s(z))\\
&=m_t(z)m_s(\varphi_t(z))\varphi_s(\varphi_t(z))=
m_{t+s}(z)\varphi_s(\varphi_t(z)).
\end{align*}
Thus, $m_t(z)\neq0$ implies $\varphi_{t+s}(z)=\varphi_s(\varphi_t(z))$. In addition, since $(T_t)_{t\geq0}$ is strongly continuous, it is easy to deduce that $(t,z)\mapsto\varphi_{t}(z)$ is continuous on $[0,\infty)\times \mathbb{D}$. By the Vitali's Theorem, it follows that for each $z\in\mathbb{D}$, $\varphi_{t}(z)\rightarrow z$, as $t\rightarrow0^+$. That is, we prove that $\at$ is a semigroup of analytic functions from $\D$ to $\D$. Therefore, the above results results: $m_0\equiv0$, $m_t(z)=(T1)(z)$ and $m_{s+t}=m_t(m_s\circ\varphi_t)$ imply that $(m_t(z))_{t\geq0}$ is a cocycle for $\at$. Moreover,
$$
T_tf(z)=m_t(z)(f\circ\varphi)(z), \quad f\in A^p_{\alpha},
$$
where $m_t(z)=(T_t1)(z)$. The proof is complete.
\end{proof}

\end{document}